\documentclass [12pt,a4paper,reqno]{amsart}
 \input amssymb.sty
\DeclareMathOperator{\sign}{sign}

\newcommand{\ef}{\end{equation}}
\chardef\bslash=`\\ 

\newtheorem{thm}{Theorem} [section]
\newtheorem*{thm*}{Theorem}

\newtheorem{lem}[thm]{Lemma}

 \theoremstyle{definition}
\newtheorem{defn}[thm]{Definition}
\newtheorem*{examp*}{Example}

 \theoremstyle{remark}

 \newtheorem{remark}[thm]{Remark}
 \newtheorem{example}[thm]
 {Example}

\newtheorem*{acknowledgment*} {Acknowledgment}
\newcommand{\thmref}[1]{Theorem~\ref{#1}}

\newcommand{\lemref}[1]{Lemma~\ref{#1}}

 \renewcommand{\sectionmark}[1]{}

 \newcommand{\supp} {\operatorname{supp}}
\newcommand{\loc} {\operatorname{loc}}

 \date{}

\begin{document}

\title[Correct solvability, embedding theorems and separability] {Correct solvability, embedding theorems and separability for  the Sturm-Liouville equation}
\author[N.A. Chernyavskaya]{N.A. Chernyavskaya}
\address{Department of Mathematics and Computer Science, Ben-Gurion
University of the Negev, P.O.B. 653, Beer-Sheva, 84105, Israel}

\author[L.A. Shuster]{L.A. Shuster}
\address{Department of Mathematics,
 Bar-Ilan University,  52900 Ramat Gan, Israel} 
 \email{miriam@macs.biu.ac.il
 }

 \subjclass[2010] {Primary 46E35; Secondary 34B24}
 \keywords{Sobolev space, embedding theorem, Sturm-Liouville equation}

\begin{abstract}
Let $p\in[1,\infty),$ $f\in L_p(\mathbb R)$ and $0\le q\in L_1^{\loc}(\mathbb R)$.
We consider the equation
\begin{equation}\label{1}
- y''(x)+q(x)y(x)=f(x),\quad x\in \mathbb R
\end{equation}
and the weighted function space
\begin{equation}\label{2}
S_p^{(2)}(R,q)=\{y\in AC_{\loc}^{(1)}(\mathbb R):\|y''-qy\|_p+\|q^{1/p}y\|_p<\infty\}.
\end{equation}
We show that there exists an embedding $S_p^{(2)}(R,q)\hookrightarrow L_p(\mathbb R)$ if and only if equation \eqref{1} is correctly solvable in $L_p(\mathbb R).$
\end{abstract}

\maketitle

\baselineskip 20pt

\section{Introduction}\label{introduction}
\setcounter{equation}{0} \numberwithin{equation}{section}

In the present paper, we continue the investigation started in \cite{1},\cite{2}.
 We consider the equation
\begin{equation}\label{1.1}
- y''(x)+q(x)y(x)=f(x),\quad x\in \mathbb R
\end{equation}
where $f\in L_p$ ($L_p(\mathbb R):=L_p),$ $p\in [1,\infty)$ and
\begin{equation}\label{1.2}0\le q\in
L_1^{\loc}(\mathbb R).\end{equation}

  In the sequel, for brevity, we do not include \eqref{1.2} in our statements. Let $AC_{\loc}^{(1)}(\mathbb R)$ be the set of functions absolutely continuous together with their derivatives on every finite interval of the real axis. Under a solution of \eqref{1.1} we mean any function $y\in AC_{\loc}^{(1)}(\mathbb R)$ satisfying \eqref{1.1} almost everywhere on $\mathbb R.$ In addition, we say that for a given $p\in[1,\infty)$ equation \eqref{1.1} is correctly solvable in $L_p$ if (see \cite{1}):
  \begin{enumerate}\item[I)] for any function $f\in L_p$, there exists a unique solution $y\in L_p$ of \eqref{1.1};
  \item[II)] there is an absolute constance $c(p)\in (0,\infty)$ such that the solution of \eqref{1.1} $y\in L_p$ satisfies the inequality
      \begin{equation}\label{1.3}
      \|y\|_p\le c(p)\|f\|_p,\qquad \forall f\in L_p\qquad (\|f\|_{L_p}:=\|f\|_p).
      \end{equation}
      \end{enumerate}

      Precise requirements to the coefficient $q$ which guarantee I)--II) to hold are given in \cite{1} (see \thmref{thm1.1} below). For brevity, in the sequel, we say ``problem I)--II)'' or ``question on\linebreak I)--II)''. By $c$, $c(\cdot)$, we denote absolute positive constants which are not essential for exposition and may differ even within a single chaing of computations. Note that in \cite{2} for $p=1$ we suggest another approach to the problem I)--II). To state the main result of \cite{2}, let us introduce the Sobolev weighted space
      \begin{equation}\label{1.4}
      W_p^{(2)}(R,q)=\big\{y\in AC_{\loc}^{(1)}(\mathbb R): \|y''\|_p+\|qy\|_p<\infty\big\}.
      \end{equation}

\begin{thm}\label{thm1.1} \cite{2} There exists an embedding $W_1^{(2)}(R,q)\hookrightarrow L_1$ if and only if equation \eqref{1.1} is correctly solvable in $L_1$, i.e. (see \cite{1}), if and only if there is $a\in(0,\infty)$ such that $m(a)>0$. Here
\begin{equation}\label{1.5}
    m(a)=\inf_{x\in\mathbb R}\int_{x-a}^{x+a}q(t)dt.
      \end{equation}
        \end{thm}

      In the present paper, we find an analogue of \thmref{thm1.1} in the case $p\in(1,\infty).$ To state our results, let us introduce another weighted space $S_p^{(2)}(R,q),$ $p\in[1,\infty):$
      \begin{equation}\label{1.6}
     S_p^{(2)}(R,q)=\{y\in AC_{\loc}^{(1)}(\mathbb R):\|y"-qy\|_p+\|q^{1/p}y\|_p<\infty\}.
      \end{equation}

      Our main result is the following theorem.

      \begin{thm}\label{thm1.2}
      For $p\in[1,\infty),$ there exists an embedding $S_p^{(2)}(R,q)\hookrightarrow L_p$ if and only if equation \eqref{1.1} is correctly solvable in $L_p$, i.e., if and only if $m(a)>0$ for some $a\in(0,\infty)$ (see \cite{1} and \eqref{1.5}).
      \end{thm}

      Let us analyze \thmref{thm1.2}. We need the following definition.

      \begin{defn}\label{defn1.3} Normed vector spaces $B_1$ and $B_2$ that coincide set-theoretically and have equivalent norms are called \textit{isomorphic}.
      \end{defn}

      We write $B_1\sim B_2$ if the spaces $B_1$ and $B_2$ are isomorphic, and $B_1\not\sim B_2$ otherwise. In the following statement, we show that if equation \eqref{1.1} is ``good'', then $W_1^{(2)}(R,q)\sim S_1^{(2)}(R,q).$

      \begin{thm}\label{thm1.4}
      Suppose that equation \eqref{1.1} is correctly solvable in $L_1.$ Then
       \begin{equation}\label{1.7}
    S_1^{(2)}(R,q)\sim W_1^{(2)}(R,q).
      \end{equation}
      \end{thm}
      Note that for $p\in(1,\infty)$, there is no complete analogue for \thmref{thm1.4} (see below). To explain this, we need another definition.

       \begin{defn}\label{defn1.5} \cite{3} Let $p\in[1,\infty)$ and suppose that equation \eqref{1.1} is correctly solvable in $L_p.$ We say that it is separable in $L_p$ if there is an absolute constant $c(p)\in(0,\infty)$ such that the solution $y\in L_p$ of \eqref{1.1} satisfies the inequality
        \begin{equation}\label{1.8}
   \|y"\|_p+\|qy\|_p\le c(p)\|f\|_p,\qquad \forall f\in L_p.
      \end{equation}
      \end{defn}

      In the sequel, we say for brevity ``question (problem) on the separability of \eqref{1.1} in $L_p$''.  Note that the question on the separability of \eqref{1.1} was first studied in \cite{6},\cite{7} for $p=2$ (in the language of differential operators). See \cite{9}, \cite{3}, \cite{4} and the references therein. We conclude our analysis of \thmref{thm1.2} with the following assertion.

      \begin{thm}\label{thm1.6}
      Let $p\in(1,\infty)$ and suppose that equation \eqref{1.1} is correctly solvable in $L_p.$ Then
      \begin{equation}\label{1.9}
      S_p^{(2)}(R,q)\sim W_p^{(2)}(R,q)
      \end{equation}
      if and only if \eqref{1.1} is separable in $L_p.$
      \end{thm}

      Note in addition that in view of \thmref{thm1.6}, relation \eqref{1.7} is explained by the fact that every equation \eqref{1.1} which is correctly solvable in $L_p,$ $p\in[1,\infty)$ is separable in $L_1$ (Grinshpun-Otelbaev's theorem (see \cite{8}) -- the statement given her can be found in \cite{5}). In addition, for $p>1,$ the requirement on the separability of equation \eqref{1.1} in $L_p$ becomes essential because for every $p\in(1,\infty)$ there exist equations \eqref{1.1} which are correctly solvable but not separable in $L_p$ (see \cite{4}).

      \begin{example}\label{examp1.7}
      \cite{1} Equation \eqref{1.1} with
      $$q(x)=1+\cos(|x|^\theta),\qquad x\in\mathbb R, \quad \theta >0$$
      is correctly solvable in $L_p,$ $p\in(1,\infty)$ if and only if $\theta\ge 1.$ Note that for $p\in[1,\infty),$ there exists an embedding $S_p^{(2)}(R,q)\hookrightarrow L_p$ if and only if $\theta\ge1.$ Since in this case $q(x)\in[0,2]$ for $x\in\mathbb R,$ for $\theta\ge1$ equation \eqref{1.1} is also separable in $L_p.$ Hence, in this case, $S_p^{(2)}(R,q)\sim W_p^{(2)}(R,q)$ for every $p\in[1,\infty).$
      \end{example}

      \begin{example}\label{examp1.8}
      \cite{4} Let $p\in(1,\infty)$ and let $\alpha,\beta$ be positive numbers such that $\beta<\alpha<p\beta$. Set
      \begin{equation}\label{1.10} q(x)=\begin{cases} 1,\quad x\notin\bigcup\limits_{n=2}^\infty\omega_n,\quad \omega_n=[n-n^{-\alpha},n+n^{-\alpha}],\quad n\ge 2\\
      n^\beta,\quad x\in\omega_n,\quad n\ge 2\end{cases}\end{equation}
      In this case
      $$\inf_{x\in\mathbb R} q(x)=1>0,$$
      and therefore equation \eqref{1.1} with coefficient $q$ from \eqref{1.10} is correctly solvable in $L_s,$ $s\in[1,\infty)$ (see \thmref{thm1.1} above). In \cite{4} it is shown that such an equation is not separable in $L_p.$ Therefore
      $$S_p^{(2)}(R,q)\hookrightarrow L_p,\quad \text{but} \quad S_p^{(2)}(R,q)\not\sim W_p^{(2)}(R,q).$$
      \end{example}

      \section{Preliminaries}
Below we present some facts needed for the proofs (see \S3).

\begin{lem}\label{lem2.1} \cite{5}
Suppose that the following conditions holds:
\begin{equation}\label{2.1}
\int_{-\infty}^x q(t)dt>0,\qquad \int_x^\infty q(t)dt>0\qquad \forall x\in\mathbb R.
\end{equation}
Then the equation
\begin{equation}\label{2.2}
z''(x)=q(x)z(x),\qquad x\in\mathbb R
\end{equation}
has a fundamental system of solutions (FSS) $\{u,v\}$ with the following properties:
\begin{equation}\label{2.3}
u(x)>0,\qquad v(x)>0,\qquad u'(x)\le0,\qquad v'(x)\ge0,\qquad x\in\mathbb R,\end{equation}
$$v'(x)u(x)-u'(x)v(x)=1,\qquad x\in \mathbb R,$$
\begin{equation}\label{2.4}
\lim_{x\to-\infty}\frac{v(x)}{u(x)}=\lim_{x\to\infty}\frac{u(x)}{v(x)}=0.
\end{equation}
\end{lem}

\begin{lem}\label{lem2.2}
Suppose that \eqref{2.1} holds. For a given $x\in\mathbb R,$ consider the equation in $d\ge0:$
\begin{equation}\label{2.5}
d\int_{x-d}^{x+d}q(t)dt=2.
\end{equation}
Equation \eqref{2.5} has a unique finite positive solution (denoted by $d(x),$ $x\in \mathbb R).$
\end{lem}

\begin{remark} \label{rem2.3}
The function $d(x),$ $x\in\mathbb R$ was introduced by M. Otelbaev (see \cite{9}).
\end{remark}

\begin{thm}\label{thm2.4} \cite{1}
Equation \eqref{1.1} is correctly solvable in $L_p,$ $p\in[1,\infty)$ if and only if \eqref{2.1} holds and $d_0<\infty.$ Here
\begin{equation}\label{2.6}
d_0=\sup_{x\in\mathbb R}d(x).
\end{equation}
\end{thm}

\begin{thm}\label{thm2.5} \cite{1}
Suppose that equation \eqref{1.1} is correctly solvable in $L_p,$ $p\in[1,\infty)$  and  let $y\in L_p$ be its solution. Then the following inequality holds:
\begin{equation}\label{2.7}
 \|q^{1/p}y\|_p\le c(p)\|f\|_p,\qquad \forall f\in L_p.
\end{equation}
\end{thm}

Note that in \S3, in the course of expositions, we present some technical assertions.

\section{Proofs}
In this section, we present the proofs of all our statements.

\renewcommand{\qedsymbol}{}
\begin{proof}[Proof of \thmref{thm1.2}]  Necessity.
\end{proof}

We need some auxiliary assertions.

\begin{lem}\label{lem3.1} If $S_p^{(2)}(R,q)\hookrightarrow L_p,$ $p\in(1,\infty)$, then \eqref{2.1} holds.
\end{lem}
\renewcommand{\qedsymbol}{\openbox}
\begin{proof} Suppose that \eqref{2.1} does not hold. Suppose, say, that there exists $x_0\in\mathbb R$ such that
\begin{equation}\label{3.1}
\int_{x_0}^\infty q(t)dt=0.
\end{equation}
Let us introduce a cutting function $\varphi(\cdot)\in C_0^\infty(\mathbb R)$ such that $\supp \varphi=[-4^{-1},4^{-1}],$ $\varphi(x)\equiv 1$ for $|x|\le 8^{-1}$ and $0\le \varphi(x)\le 1$ for $8^{-1}\le|x|\le 4^{-1}$. Consider the sequence $\{y_n(x)\}_{n=2}^\infty,$ $x\in\mathbb R$ where
\begin{equation}\label{3.2}
y_n(x)=\varphi\left(\frac{x-n}{n}\right),\quad x\in\mathbb R,\quad n>n_0=\max\left\{1,\, \frac{4}{3}x_0\right\}.
\end{equation}
It is easy to see that
\begin{equation}\label{3.3}
q(x)y_n(x)=0\qquad\text{for}\qquad x\in\mathbb R,\quad n>n_0.
\end{equation}
Indeed, for $x\le x_0$, \eqref{3.3} follows from the obvious equality
$$(-\infty,x_0]\cap\supp y_n=\emptyset,\qquad n>n_0$$
and for $x\ge x_0$, \eqref{3.3} follows from \eqref{1.2} and \eqref{3.1}. Further, by \eqref{3.3}, we have
$$-y_n''(x)+q(x)y_n(x)=-y_n''(x),\qquad x\in \mathbb R,\quad n>n_0.$$
We have the following estimates:
$$\|y_n''\|_p^p=\int_{-\infty}^\infty\left|\varphi\left(\frac{x-n}{n}\right)''\right|^pdx=\int_
{\frac{3}{4}n}^{\frac{5}{4}n}\left|\left[
 \varphi\left(\frac{x-n}{n}\right)\right]''\right|^pdx\le \frac{c}{n^{2p-1}},$$
$$\int_{-\infty}^\infty\left| q^{\frac{1}{p}}(x)y_n(x)\right|^pdx=\int_{-\infty}^{x_0}q(x)\cdot Odx+\int_{x_0}^\infty O\cdot|y_n(x)|^pdx=0,$$
$$\|y_n\|_p^p=\int_{-\infty}^\infty|y_n(x)|^pdx\ge\int_{\frac{3}{4}n}^{\frac{5}{4}n}|y_n(x)|^pdx\ge\int_{\frac{7}{8}n}^
{\frac{9}{8}n}|y_n(x)|^pdx=\frac{n}{4}.$$
Hence $y_n\in S_p^{(2)}(R,q)$ for $n\ge n_0.$ Therefore, from the embedding $S_p^{(2)}(R,q)\hookrightarrow L_p$, we obtain
$$\left(\frac{n}{4}\right)^{\frac{1}{p}}\le \|y_n\|_p\le C\{\|y_n''-qy_n\|_p+\|q^{\frac{1}{4}}y_n\|_p\}\le \frac{C}{n^{2-\frac{1}{p}}}\quad\Rightarrow\quad n^2\le C<\infty,\quad n>n_0,$$
contradiction.
\end{proof}

Thus, from Lemmas \ref{lem3.1} and \ref{lem2.2}, it follows that for every $x\in\mathbb R$ the function $d(x)$ is well-defined. Now, for a given $x\in\mathbb R$, consider the following Cauchy problem:
\begin{equation}\label{3.4}
y''(t)=q(t)y(t),\qquad t\in\mathbb R,
\end{equation}
\begin{equation}\label{3.5}
y (t)\big|_{t=x}=1,\qquad y'(t)\big|_{t=x}=0.
\end{equation}

\begin{lem}\label{lem3.2} \cite{1} Under conditions \eqref{1.2} and \eqref{3.1}, for the solution $y$ of problem \eqref{3.4}--\eqref{3.5}, we have the relations
\begin{equation}\label{3.6}
y (t)\ge1,\qquad \sign y'(t)=\sign(t-x),\qquad t\in\mathbb R,
\end{equation}
\begin{equation}\label{3.7}
1\le y(t)\le 4\qquad\text{for}\qquad t\in\left[x-\frac{d(x)}{4},\, x+\frac{d(x)}{4}\right],
\end{equation}
\begin{equation}\label{3.8}
d(x)|y'(t)|\le 8 \qquad\text{for}\qquad t\in\left[x-\frac{d(x)}{4},\, x+\frac{d(x)}{4}\right].
\end{equation}
\end{lem}

\begin{lem}\label{lem3.3} \cite{1} Suppose that conditions \eqref{1.2} and \eqref{3.1} hold. Let $\varphi$ be the cutting function from \lemref{lem3.1}, let $y$ be the solution of problem \eqref{3.4}--\eqref{3.5}, and let
\begin{equation}\label{3.9}
z(t)=d(x)^{2-\frac{1}{p}}\left[\varphi\left(\frac{t-x}{d(x)}\right)\right]\cdot y(t),\qquad t\in R.
\end{equation}
Then $z(t)$ is a solution of equation \eqref{1.1} with $f=f_1$
\begin{equation}\label{3.10}
f_1(t)=-d(x)^{2-\frac{1}{p}}\left[\varphi\left(\frac{t-x}{d(x)}\right)\right]''\cdot y(t)-2d(x)^{2-\frac{1}{p}}\cdot\left[\varphi\left(\frac{t-x}{d(x)}\right)\right]'\cdot y'(t),\ t\in\mathbb R.
\end{equation}
In addition, we have the estimates
\begin{equation}\label{3.11}
\|f_1\|_p\le C,\qquad x\in\mathbb R,
\end{equation}
\begin{equation}\label{3.12}
c^{-1}d(x)^2\le \|z\|_p\le c d(x)^2,\qquad x\in\mathbb R.
\end{equation}
\end{lem}

\begin{lem}\label{lem3.4}
Under conditions \eqref{1.2} and \eqref{3.1}, for the function $z$ (see \eqref{3.9}) we have the estimate
\begin{equation}\label{3.13}
\|q^{\frac{1}{p}}z\|_p\le cd^{\frac{2}{p'}}(x),\qquad p'=\frac{p}{p-1},\qquad x\in\mathbb R.
\end{equation}
\end{lem}

\begin{proof}
Below we successively use \eqref{3.9}, \eqref{3.6} and \eqref{2.5}:
\begin{align*}
\|q^{\frac{1}{p}}z\|_p^p&=\int_{-\infty}^\infty q(t)|z(t)|^pdt =\int_{x-\frac{d(x)}{4}}^{x+\frac{d(x)}{4}}q(t)d(x)^{2p-1}
\left|\varphi\left(\frac{t-x}{d(x)}\right)\right|^p|y(t)|^pdt\\
&\le cd(x)^{2p-1}\int_{x-d(x)}^{x+d(x)}q(t)dt=cd(x)^{2p-2}.
\end{align*}
\end{proof}

 Thus, since $z$ is a solution of \eqref{1.1} with $f=f_1,$ from \eqref{3.11} and \eqref{3.13} it follows that $z\in S_p^{(2)}(R,q)$. Then from the embedding $S_p^{(2)}(R,q)\hookrightarrow L_p,$ using \eqref{3.12}, \eqref{3.11} and \eqref{3.13}, we obtain
 \begin{equation}\label{3.14}
 d^2(x)\le c\|z\|_p\le c\big(\|z''-qz\|_p+\big\|q^{\frac{1}{p}}z\big\|_p\big)\\
 \le c(1)+d(x)^{\frac{2}{p'}},\qquad x\in\mathbb R.
 \end{equation}
 Since $2>\frac{2}{p'}$, from \eqref{3.14} it follows that $d_0<\infty$ (see \eqref{2.6}). It remains to refer to \thmref{thm2.4}.

 \begin{remark}\label{rem3.5} Test functions \eqref{3.9} and Lemmas \ref{lem3.2} and \ref{lem3.3} were first used (for different purposes) in \cite{10}.
 \end{remark}

\begin{proof}[Proof of \thmref{thm1.2}]  Sufficiency.

Suppose that for $p\in(1,\infty)$ equation \eqref{1.1} is correctly solvable in $L_p.$ Let us show that $S_p^{(2)}(R,q)\hookrightarrow L_p.$ Let $y\in S_p^{(2)}(R,q).$ Denote
\begin{equation}\label{3.15}
f_0(x)=-y''(x)+q(x)y(x),\qquad x\in\mathbb R.
\end{equation}
Clearly, $f_0\in L_p$ (see \eqref{1.6}). Consider the equation
\begin{equation}\label{3.16}
-\tilde y''(x)+q(x)\tilde y(x)=f_0(x),\qquad x\in R.
\end{equation}
By I)--II), there exists a unique solution $\tilde y\in L_p$ of \eqref{3.16}. Then
\begin{gather}
-\tilde y''(x)+q(x)\tilde y(x)=f_0(x)=-y''(x)+q(x)y(x),\qquad x\in\mathbb R\quad\Rightarrow\nonumber\\
z''(x)=q(x)z(x),\qquad z(x)=\tilde y(x)-y(x),\qquad x\in\mathbb R.\label{3.17}
\end{gather}

Note that \eqref{2.7} and \eqref{1.6} imply, respectively, the inequalities
 \begin{gather}
\big\|q^{\frac{1}{p}}\tilde y\big\|_p\le c\|f_0\|_p<\infty,\qquad \big\|q^{\frac{1}{p}}y\big\|_p<\infty\quad\Rightarrow\nonumber\\
\big\|q^{\frac{1}{p}}z\big\|_p=\big\|q^{\frac{1}{p}}(\tilde y-y)\big\|_p\le\big\|q^{\frac{1}{p}}\tilde y\big\|_p+\big\|q^{\frac{1}{p}}y\big\|_p<\infty.\label{3.18}
\end{gather}
On the other hand, from \eqref{3.17} and \lemref{lem2.1}, we obtain
 $$z(x)=c_1u(x)+c_2v(x),\qquad x\in\mathbb R.$$

Let us show that $c_1=c_2=0.$ Let, say, $c_2\ne0.$ Then (see \eqref{2.4}) there exists $x_0\gg1$ such that
\begin{equation}\label{3.19}
\left|\frac{c_1}{c_2}\right|\frac{u(x)}{v(x)}\le \frac{1}{2}\qquad\text{for}\qquad x\ge x_0.
\end{equation}
Now from \eqref{3.17}, \eqref{3.18}, \eqref{3.19}, \eqref{2.3} and the inequality $m(a)>0,$ $a\in (0,\infty)$ (see \thmref{thm1.2}), we obtain
\begin{align*}
\infty&>\big\|q^{\frac{1}{p}}z\big\|_p>\int_{x_0}^\infty q(x)|c_1u(x)+c_2v(x)|^pdx\\
&\ge |c_2|^p\int_{x_0}^\infty q(x)v(x)^p\left|1-\left|\frac{c_1}{c_2}\right|\frac{u(x)}{v(x)}\right|^pdx\\
&\ge\left|\frac{c_2}{2}\right|^pv(x_0)^p\int_{x_0}^\infty q(t)dt=\infty.
\end{align*}
Contradiction. Hence $c_2=0$ and, similarly,$c_1=0.$ Thus $\tilde y=y.$ Therefore $y\in L_p,$ and using II) we obtain
$$\|y\|_p=\|\tilde y\|_p\le c(p)\|\tilde y''-q\tilde y\|_p=c(p)\|f_0\|_p=c(p)\|y''-qy\|_p\le c(p)\big(\|y''-qy\|+\big\|q^{\frac{1}{p}}\big\|_p\big).$$
Here $c(p)$ is an absolute constant from \eqref{1.3}. We conclude that $S_p^{(2)}(R,q)\hookrightarrow L_p$ and, in addition, from \cite{1} it follows that this embedding holds if and only if $m(a)>0$ for some $a\in(0,\infty)$ (see \eqref{1.5}).
 \end{proof}

 \begin{proof}[Proof of \thmref{thm1.4}]
 Since equation \eqref{1.1} is correctly solvable in $L_1,$ from the embeddings $W_1^{(2)}(R,q)\hookrightarrow L_1,$ $ S_1^{(2)}(R,q)\hookrightarrow L_1$ (see Theorems \ref{thm1.1} and \ref{thm1.2}), it follows that $W_1^{(2)}(R,q)$ and $S_1^{(2)}(R,q)$ are normed spaces with norms
 $$\|y\|_{W_1^{(2)}(R,q)}=\|y''\|_1+\|qy\|_1,\qquad \|y\|_{S_1^{(2)}(R,q)}=\|y''-qy\|_1+\|qy\|_1,$$
 respectively. Relation \eqref{1.7} follows from the obvious inequalities based on the triangle inequality for norms:
 $$\frac{1}{2}\|y\|_{W_1^{(2)}(R,q)}\le \|y\|_{S_1^{(2)}(R,q)}\le 2\|y\|_{W_1^{(2)}(R,q)}.$$
 \end{proof}

 \renewcommand{\qedsymbol}{}
\begin{proof}[Proof of \thmref{thm1.6}]  Necessity.

 Suppose that for $p\in(1,\infty)$ equation \eqref{1.1} is correctly solvable and \eqref{1.9} holds. We check that \eqref{1.8} holds. Let $y$ be a solution of \eqref{1.1} belonging to the class $L_p.$ Then $y\in S_p^{(2)}(R,q)$ by \thmref{thm1.2}, and, in addition, \eqref{2.7} holds. Together with \eqref{1.9}, this implies that
 $$\|y''\|_p+\|qy\|_p\le c(p)\big(\|y''-qy\|_p+\|q^{\frac{1}{p}}y\|_p\big)\le 2c(p)\|y''-qy\|_p\ \Rightarrow\ \eqref{1.8}.$$
\end{proof}

\renewcommand{\qedsymbol}{\openbox}

\begin{proof}[Proof of \thmref{thm1.6}]  Sufficiency.
Suppose that for some $p\in(1,\infty)$ equation \eqref{1.1} is correctly solvable and separable in $L_p$. Let us check \eqref{1.7}. By \thmref{thm1.2}, for any function $f\in L_p,$ the solution $y\in L_p$ of \eqref{1.1} belongs to the space $S_p^{(2)}(R,q).$   Since \eqref{1.8} holds, we have $y\in W_p^{(2)}(R,q).$ Further, if $y\in W_p^{(2)}(R,q),$ then $(y''-qy)\in L_p$ by the triangle inequality, and $q^{\frac{1}{p}}\in L_p$ by \eqref{2.7}. Hence $y\in S_p^{(2)}(R,q).$ Thus $S_p^{(2)}(R,q)=W_p^{(2)}(R,q)$ set-theoretically. Finally, from  correct solvability, separability and \eqref{2.7}, we obtain
\begin{gather*}
\|y''-qy\|_p\le\|y''\|_p+\|qy\|_p,\qquad y\in S_p^{(2)}(R,q)\\
\big\|q^{\frac{1}{p}}y\big\|_p\le c(p)\|y''-qy\|_p\le c(p)(\|y''\|_p+\|qy\|_p)\\
\|y''-qy\|_p+\big\|q^{\frac{1}{p}}y\big\|_p\le (c(p)+1)(\|y''\|_p+\|qy\|_p).
\end{gather*}

The converse inequality follows from \eqref{1.8}:
$$\|y''\|_p+\|qy\|_p\le c(p)(\|y''-qy\|_p)\le c(p)\big(\|y''-qy\|_p+\big\|q^{\frac{1}{p}}y\big\|_p\big).$$
\end{proof}


\begin{thebibliography}{10}

 \bibitem[1]{4}  N. Chernyavskaya and L. Shuster,
\emph{Weight summability of solutions of the Sturm-Liouville equation},   J. Diff. Eq. \textbf{151} (1999), 456-473.

 \bibitem[2]{5}  N. Chernyavskaya and L. Shuster,
\emph{Estimates for the Green function of a general
Sturm-Liouville operator and their applications}, Proc. Amer.
Math. Soc. \textbf{127} (1999), 1413-1426.


\bibitem[3]{1} N. Chernyavskaya and L. Shuster, \emph{A criterion for correct solvability
of the Sturm-Liouville equation in the space $L_p(R)$}, Proc.
Amer. Math. Soc \textbf{130}
  (2001), 1043-1054.

  \bibitem[4]{2} N. Chernyavskaya and L. Shuster, \emph{An embedding theorem for a weighted space of Sobolev type and correct solvability of the Sturm-Liouville equation}, Czechoslovak Math. J.   \textbf{62} (2012), 709-716.

\bibitem[5]{3} N. Chernyavskaya and L. Shuster, \emph{Weighted estimates for solutions of the general Sturm-Liouville equation and the Everitt-Giertz problem I}, Proc. Edinburgh Math. Soc., to appear.


 \bibitem[6]{6} W.N. Everitt and M. Giertz, \emph{Some properties of the domains of certain differential operators}, Proc.
London Math. Soc. \textbf{23} (1971), no. 3, 301-324.


\bibitem[7]{7} W.N. Everitt and M. Giertz, \emph{Some inequalities associated with certain differential operators}, Math. Z.
\textbf{126} (1972), no. 4, 308-326.

\bibitem[8]{8} E. Grinshpun and M. Otelbaev, \emph{On smoothness of solutions of a nonlinear Sturm-Liouville
equation in $L_1(-\infty,\infty)$}, Izvestiya Akad. Nauk Kazach.
SSR \textbf{5} (1984), 26-29.

\bibitem[9]{9} K. Mynbaev and M. Otelbaev, \emph{Weighted Function Spaces and the
Spectrum of Differential Operators}, Nauka, Moscow, 1988.


\bibitem[10]{10} M. Otelbaev, \emph{On smoothness of   solutions of differential equations}, Izv. Akad. Nauk. Kazakh SSR \textbf{5}
(1977), 45-48.
\end{thebibliography}
\end{document}